\documentclass[11pt]{article}
\usepackage{bm}
\usepackage{tikz}
\usepackage[hidelinks]{hyperref}
\usepackage{indentfirst}
\usepackage{amsthm,amsmath,amssymb,mathrsfs,graphicx,color,url}

\theoremstyle{plain}
\newtheorem{thm}{Theorem}[section]
\newtheorem{prop}[thm]{Proposition}
\newtheorem{lem}[thm]{Lemma}
\newtheorem{cor}[thm]{Corollary}
\theoremstyle{definition}

\newtheorem{ex}[thm]{Example}

\newtheorem{rem}[thm]{Remark}

\numberwithin{equation}{section}

\makeatletter
\def\filename{\texttt{\jobname.tex}} 
\makeatother

\newcommand{\newoperator}[2]{\DeclareMathOperator{#1}{#2}}
\newoperator{\var}{var}
\newoperator{\cov}{cov}

\def\E{\mathbb{E}}
\def\P{\mathbb{P}}
\def\R{\mathbb{R}}
\def\Z{\mathbb{Z}}
\def\inlaw{\stackrel{d}{=}}

\title{A remark on elephant random walks via the classical
law of the iterated logarithm for self-similar Gaussian
processes} 
\author{Shuhei Shibata\thanks{Joint Graduate School of Mathematics for Innovation, Kyushu
   University, Fukuoka, 819-0395, JAPAN. \\
   \textit{E-mail address}: \texttt{shibata.shuhei.746@s.kyushu-u.ac.jp}} and Tomoyuki Shirai\thanks{Institute of Mathematics for Industry, Kyushu
   University, Fukuoka, 819-0395, JAPAN. 
   \textit{E-mail address}: {\texttt{shirai@imi.kyushu-u.ac.jp}}}
   }
\date{}
\begin{document}
\maketitle
\begin{abstract}
This paper investigates whether two independent Elephant Random Walks (ERWs) on $\Z$, each with a different memory parameter, can meet infinitely often, extending the work of Roy, Takei, and Tanemura. We also study the asymptotic behavior of their distance by providing an elementary and accessible proof of the classical Law of the Iterated Logarithm (LIL) for centered, continuous, self-similar Gaussian processes under a certain decay condition on the covariance kernel. 
\end{abstract}

\section{Introduction}

The asymptotic behavior of random walks with long-range
memory has been extensively studied in recent years. Among
such models, the Elephant Random Walk (ERW), introduced by
Sch\"{u}tz and Trimper \cite{schutz2004elephants}, has attracted significant attention
as a non-Markovian stochastic process that 
exhibits a phase transition from diffusive to superdiffusive behavior
(cf. \cite{baur2016elephant, bercu2017martingale, 
coletti2017central, coletti2017strong,
kubota2019gaussian, guerin2023fixed, qin2023recurrence}). 
In this paper, we investigate the interaction between two
independent ERWs on the integer lattice
$\Z$. Specifically, we focus on the phenomenon of
infinite collisions, inspired by the result of Roy, Takei,
and Tanemura \cite{roy2024often}, who studied the case where both ERWs have
the same memory parameter. Our work extends their result by
considering different memory parameters for the two
walks. In addition, we obtain asymptotic results for the
distance between the two ERWs. 

We recall the definition of the one-dimensional ERW. 
The first step $X_1$ of the elephant is $+1$ with
probability $q$, and $-1$ with probability $1-q$. 
For each $n = 1,2,\dots$, given the past steps $X_1,X_2,\dots, X_n$, we define the
$(n+1)$-th step by 
\begin{equation*}
    X_{n+1} =
    \begin{cases}
      +X_{U_n} & \text{with probability} \quad p, \\
      -X_{U_n} & \text{with probability} \quad 1-p, 
    \end{cases}
\end{equation*}
where $p \in [0,1]$ is called the \textit{memory parameter}, and $U_n$ is uniformly
distributed on $\{1,2,\dots,n\}$ and independent of
$\{X_i\}_{i=1}^n$. 
The sequence $\{X_i\}_{i=1}^{\infty}$ generates the
one-dimensional ERW $\{{S_n}\}_{n=0}^{\infty}$ defined by 
\begin{equation*}
    S_0 := 0, \quad\text{and}\quad S_n = \sum_{i=1}^n X_i
     \quad\text{for} \quad n=1,2,\dots. 
\end{equation*}

As is well known, the asymptotic behavior of the ERW mainly
depends on the memory parameter $p$ which takes the values
from $0$ to $1$. In the diffusive regime $p<3/4$ as well as
in the critical regime $p=3/4$, the following asymptotic normality of
the ERW is established: as $n\to \infty$,
\begin{equation*}
    \frac{S_n}{\sqrt{n}}\xrightarrow{d} N\left(0,\frac{1}{3-4p}\right) \quad\text{for} \quad 0<p<3/4
\end{equation*}
and
\begin{equation*}
    \frac{S_n}{\sqrt{n\log n}}\xrightarrow{d} N\left(0,1\right) \quad\text{for} \quad p=3/4
\end{equation*}
 (cf. \cite{baur2016elephant,
bercu2017martingale,  coletti2017central}). On the other hand, in the
superdiffusive regime $p>3/4$, the limiting distribution of the ERW  is not Gaussian:
\begin{equation}\label{superdiffusive L}
    \lim_{n\to \infty} \frac{S_n}{n^{2p-1}}=L_{p,q} \quad
       \textit{a.s.},
\end{equation}
where $L_{p,q}$ has a continuous distribution which is not Gaussian. (cf. \cite{bercu2017martingale, coletti2017central,
kubota2019gaussian, guerin2023fixed}).\par
In the diffusive and critical regimes, Roy, Takei, and Tanemura \cite{roy2024often} established the
asymptotic results for the distance between two independent
ERWs with the same memory parameter on
the integer lattice $\mathbb{Z}$.
\begin{thm}[\cite{roy2024often}]\label{asymptotic behavior the distance with same parameters}
Let \(\{S_n\}_{n=0}^{\infty}\) and \(\{S'_n\}_{n=0}^{\infty}\) be two independent ERWs on $\Z$ with the same memory parameter \(p\).
\begin{itemize}
    \item[(a)] If $0 < p < 3/4$ \text{then}
    \begin{equation*}
      \limsup_{n \to \infty}\pm \frac{S_n - S'_n}{\sqrt{2 n
       \log \log n}} = \frac{\sqrt{2}}{\sqrt{3 - 4p}} \quad
       \text{a.s.}
    \end{equation*}
    \item[(b)] If $p = 3/4$ \text{then}
    \begin{equation*}
      \limsup_{n\to\infty}\pm\frac{S_n - S'_n}{\sqrt{2 n\log
       n\log\log\log n}} = \sqrt{2} \quad \text{a.s.} 
    \end{equation*}
\end{itemize}
\end{thm} 
In Theorem~\ref{asymptotic behavior the distance with same
parameters}, the memory parameters of the two elephants are
assumed to be identical. In this case, the LIL for the Brownian motion can be
applied to obtain their result, which leads to infinite collisions. However, when the memory
parameters are different, the LIL for self-similar
Gaussian processes is required instead.

Although the LIL has been extensively studied in a general
framework (cf. \cite{stout, taqqu1985survey,
bingham1986variants,  LedouxTalagrand}), we do not aim to treat it
in full generality in this paper. Rather, we present a
self-contained and transparent exposition within a
restricted yet sufficiently meaningful framework, motivated
by the analysis of infinite collisions of two independent ERWs.

\begin{thm}\label{cLIL main result}
Let $\{X(s)\}_{s\ge 0}$ be a centered continuous Gaussian
 process with covariance kernel $R(s,t)$ satisfying the
 following. 
\begin{itemize}
 \item[(i)] There exists $\rho >0$ such that for any
	    $c>0$ and $s,t >0$, 
\begin{equation}
R(cs, ct) = c^{2\rho} R(s,t). 
\label{assv}
\end{equation}
\item[(ii)] Under (i), there exists $\eta>0$  such that 
the function $h(x)=x^{-\rho}R(1,x) \ (x \ge 1)$ satisfies 
\begin{equation}\label{h order of convergence}
    h(x)=O((\log x)^{-\eta}) \text{\quad as $x\to \infty$}. 
\end{equation}
\end{itemize}
Then, we have 
\[
 \limsup_{s \to \infty}  \pm \frac{X(s)}{\sqrt{2 s^{2\rho} \log \log s}} 
= R(1,1)^{\frac{1}{2}} \quad a.s. 
\] 
\end{thm}

\begin{rem}\label{rem:stationary}
In \cite{lamperti1962semi}, Lamperti proved that if the
 process $\{X(t)\}_{t \ge 0}$ is nontrivial, stochastically continuous at
 $t=0$ and there 
 exists a function $v(c)$ such that  
    \begin{equation*}
        R(cs, ct) = v(c) R(s,t),
    \end{equation*}
then it must be of the form $v(c) = c^{2\rho}$ with $\rho \ge 0$.
Under the assumption (i), we have the
 $\rho$-self-similarity, i.e., 
\begin{equation}
\{X(ct)\}_{t \ge 0} \inlaw \{c^{\rho} X(t)\}_{t \ge 0}
 \quad (c>0)
\label{eq:rho-selfsimilar} 
\end{equation}
and $X(0)=0$ almost surely. 
Moreover, if the process $\{X(t)\}_{t\ge 0}$ 
is $\rho$-self-similar, then $\{Y(t)=e^{-\rho
 t}X(e^t)\}_{t \in \R}$ is strictly stationary. 
Conversely, if the process $\{Y(t)\}_{t \in \R}$ 
is strictly stationary, then $\{X(t)=t^{\rho} Y(\log t)
\}_{t > 0}$ with $X(0)=0$ is $\rho$-self-similar (cf. \cite{taqquSamorodnitsky, Maejima}). 
From this correspondence, if we set $r_Y(t) := \E[Y(0)Y(t)]$, then $h(x) = r_Y(\log
 x)$, and the LIL for a self-similar Gaussian process
 $\{X(t)\}_{t \ge 0}$ can be 
rephrased as $\limsup_{t \to \infty} \pm Y(t)/\sqrt{2 \log t} =
 r_Y(0)^{1/2}$ for the associated stationary Gaussian process
 $\{Y(t)\}_{t  \in \mathbb{R}}$. 
Several related works on this form of the limit theorem and on 
the version where the process $\pm Y(t)$ is replaced by 
its maximum process $\sup_{0 \le s \le t} Y(s)$ 
have been extensively studied in the continuous-parameter setting \cite{nisio1967extreme, marcus1972upper, lai1973gaussian, lai1974reproducing, Mangano1976Onstrassen, Carmona1976convergence, arcones1995law}, and also, in the discrete parameter setting 
\cite{berman1964limit, pickands1967maxima, Pickands1969,
 oodaira1971law, rio1995functional}.
\end{rem}

\begin{rem}
Our proof of Theorem~\ref{cLIL main result} relies on a version of the (second) Borel–Cantelli lemma proved by
      Erd\H{o}s and R\'{e}nyi \cite{erdHos1959cantor}, along
      with a decay condition \eqref{h order of convergence}. 
As mentioned in Remark~\ref{rem:stationary}, 
the decay condition (or mixing condition) corresponding to
\eqref{h order of convergence} on
the covariance function of \( \{Y(t)\}_{t \in \mathbb{R}} \)
is $r_Y(t) = O(t^{-\eta})$ as $t \to \infty$,  
under which the Feller's type LIL for stationary processes 
was established in \cite{qualls1971asymptotic}.  
Furthermore, this condition was relaxed to an even weaker mixing condition 
$r_Y(t) = O((\log t)^{-1})$ in \cite{pathak1973law} by using the 
zero-one law for the event $A=\{Y(t)>f(t)\ i.o.\}$,  
where $\{Y(t)\}_{t \in \R}$ is a stationary Gaussian process
      and $f(t)$ is any non-decreasing positive 
function on some time interval $[a,\infty)$ \cite{qualls1972note}. 
The functional LIL, as a natural extension, has also been thoroughly developed for self-similar Gaussian processes 
(cf. \cite{strassen1964invariance, oodaira1972strassen,oodaira1973law, kuelbs1976, taqqu1977law, kono1983, taqqu1985survey}). 
\end{rem}			

Here we give some examples of self-similar Gaussian
processes satisfying two assumptions (i) and (ii) in Theorem~\ref{cLIL main result}.  
\begin{ex}
(a) Suppose that the process $X$ is the fractional Brownian
 motion (FBM) $\{B_H(t)\}_{t \ge 0}$ with 
 the covariance kernel
        \begin{equation*}
            R(s,t)=\frac{1}{2}(|t|^{2H}+|s|^{2H}-|t-s|^{2H}) 
        \end{equation*}
for $0<H<1$. Then it satisfies \eqref{assv} with $\rho=H$,
 and 
\begin{equation*}
h(x)=O(x^{-(1-H)})\text{\quad as $x\to \infty$}, 
\end{equation*}
which implies \eqref{h order of convergence}. 
We note that for $0 < H < 1$, if the process $X= \{X(t)\}_{t \ge 0}$ is 
 an $H$-self-similar Gaussian process with \textit{stationary
 increments}, then it is necessarily the FBM up to a multiplicative constant \cite{taqqu1979self}. 

(b) Let $X(t)=\int_0^t (t-u)^{\beta} u^{-\gamma/2}dB(u)$ for
 $\beta>-1/2$ and $0 \le \gamma<1$.  
This processes with $\gamma=0$ appeared in
 \cite{oodaira1972strassen} as typical examples in the context of the functional LIL,
 while the version with general $\gamma$ were introduced in
 \cite{pang2019nonstationary} 
and is referred to as the generalized Riemann-Liouville
 FBM. For example, $X(t)$ 
reduces to the standard Brownian motion when $\beta=\gamma=0$, and to
 $X(t) = \int_0^t  B(u)du$ when $\beta=1, \gamma=0$. 
In general, the covariance kernel is given by 
        \begin{equation*}
            R(s,t)=\int_0^{s\wedge
	     t}(s-u)^{\beta}(t-u)^{\beta}
	     u^{-\gamma} du,
        \end{equation*}
        which satisfies \eqref{assv} with
 $\rho=\beta-\gamma/2+1/2$, and $h(x)$ satisfies \eqref{h
 order of convergence}. Indeed, 
        \begin{equation*}
            h(x)=O(x^{-\frac{1}{2}(1-\gamma)})\text{\quad as $x\to
	     \infty$}. 
        \end{equation*}

(c) As mentioned in Remark~\ref{rem:stationary}, 
if $\{X(t)\}_{t\ge 0}$ is a $\rho$-self-similar Gaussian process, i.e., it
 satisfies \eqref{assv} with $\rho>0$, 
then there exists a stationary Gaussian process $\{Y(t)\}_{t
 \in \R}$, and we have $\{X(t)=t^{\rho} Y(\log t)
\}_{t > 0}$ with $X(0)=0$. If $\{Y(t)\}_{t
 \in \R}$ is continuous in the mean, by Bochner's theorem, 
there exists a finite measure $\mu$ on $\mathbb{R}$
 such that $r_Y(t):=E[Y(0)Y(t)] = R(1,1) \int_{\R} e^{i\xi t}
 \mu(d\xi)$. Therefore, it is clear from the definition of
 $h(x)$ that 
        \begin{equation*}
            h(x)= R(1,1) \int_{\R} e^{i\xi\log x}\mu(d\xi).
        \end{equation*}
For example, since the probability density $f_{\alpha}(x)$ 
of the symmetric $\alpha$-stable distribution
 corresponding to $\mu_{\alpha}(d\xi)=e^{-|\xi|^{\alpha}}d\xi$ for
 $0< \alpha < 2$ is well known to satisfy $f_{\alpha}(x)
 \sim C_{\alpha}
 |x|^{-(\alpha+1)}$ as $x \to \infty$ (cf. \cite{feller1991introduction}), we see that 
 \begin{equation*}
  h(x)=O((\log x)^{-(\alpha +1)})\text{\quad as
   $x\to \infty$}. 
 \end{equation*}
For $\alpha=2$, $h(x)$ decays even faster, i.e., $h(x) =
 O(e^{-(\log x)^2/4})$ as
   $x\to \infty$. 
\end{ex}

We obtain the asymptotic behavior of the distance between two 
ERWs with different memory parameters by 
applying Theorem~\ref{cLIL main result} together with the
invariance principle for the ERW established by 
Coletti, Gava, and Sch\"{u}tz \cite{coletti2017strong} (see Lemma~\ref{lem3} below). 

\begin{thm}\label{asymptotic behavior the distance with different parameters}
   Let \(\{S_n\}_{n=0}^{\infty}\) and \(\{S'_n\}_{n=0}^{\infty}\) be two independent
 ERWs on $\Z$ with the memory parameters $p$ and
 $p'$, respectively. 
If $0 < p < 3/4$ and $0 < p' < 3/4$, then we have
\begin{equation}
\limsup_{n \to \infty}\pm \frac{S_n - S'_n}{\sqrt{2 n \log
 \log n}} = \sqrt{\frac{1}{3-4p}+\frac{1}{3-4p'}} \quad \text{a.s.} 
\label{eq:asymp_variance}
\end{equation}
\end{thm}

The following result on infinite collisions of two ERWs with
distinct memory parameters follows from the difference in their speeds, which is
derived from 
Theorems~\ref{asymptotic behavior the distance with same parameters} and \ref{asymptotic
 behavior the distance with different parameters}.

\begin{cor}\label{cor}
Two independent ERWs on $\Z$ with memory parameters $p$ and $p'$ meet each other
 infinitely often almost surely if both parameters satisfy $0 < p \le 3/4$ and $0 < p'
 \le 3/4$; otherwise they meet each other only finitely often almost surely.
\end{cor}

The remainder of the paper is organized as follows. 
In Section 2, we present a proof of Theorem~\ref{asymptotic
behavior the distance with different parameters} assuming
the LIL result stated in Theorem~\ref{cLIL main result}. In Section 3, we establish upper
and lower bounds needed to prove the LIL result, while
leaving the main ingredient of the proof in Section 4. 
The main arguments required to complete the proof of
Theorem~\ref{asymptotic behavior the distance with different
parameters} and Corollary~\ref{cor}
are carried out in Section 4.

\section{Strong asymptotic distance between two \\ERWs}

First we give a proof of Theorem~\ref{asymptotic behavior the distance with different parameters} assuming
Theorem~\ref{cLIL main result}. To prove
Theorem~\ref{asymptotic behavior the distance with different parameters}, we use the following strong asymptotic
relation between Brownian motion and ERW.

\begin{lem}[Coletti, Gava, and Sch\"{u}tz \cite{coletti2017strong}] \label{lem3}
Let \(\{S_n\}_{n=0}^{\infty}\) be the ERW on $\Z$ with the memory parameter \(p\), and \(\{B(t)\}_{t\ge 0}\) be the standard one-dimensional Brownian motion.
\begin{itemize}
    \item[(a)] If $0 < p < 3/4$ \text{then}
    \begin{equation*}
      \ S_n - \frac{n^{2p-1}}{\sqrt{3-4p}}B(n^{3-4p}) =
       o(\sqrt{n\log\log n}) \quad \text{a.s.} 
\end{equation*}
    \item[(b)] If $p = 3/4$ \text{then}
    \begin{equation*}
      \ S_n - \sqrt{n}B(\log n) = o(\sqrt{n\log
       n\log\log\log n}) \quad \text{a.s.}
    \end{equation*}
\end{itemize}
\end{lem}

Now we prove Theorem~\ref{asymptotic behavior the distance
with different parameters} by using Theorem~\ref{cLIL main
result} and Lemma~\ref{lem3}. 

\begin{proof}[Proof of Theorem~\ref{asymptotic behavior the distance with different parameters}]
Suppose that $0 < p < 3/4$ and $0 < p' < 3/4$. It follows form Lemma~\ref{lem3} that
    \begin{equation*}
      \ S_n - \frac{n^{2p-1}}{\sqrt{3-4p}}B(n^{3-4p}) =
       o(\sqrt{n\log\log n}) \quad \textit{a.s.}
    \end{equation*}
and
    \begin{equation*}
      \ S'_n - \frac{n^{2p'-1}}{\sqrt{3-4p'}}B'(n^{3-4p'}) =
       o(\sqrt{n\log\log n}) \quad \textit{a.s.,}
    \end{equation*}
where \(\{B(t)\}_{t \ge 0}\) and \(\{B'(t)\}_{t \ge 0}\) are
 two independent copies of the standard one-dimensional Brownian motion.
Therefore, we have
\begin{equation*}
    S_n - S'_n = \frac{n^{2p-1}}{\sqrt{3-4p}}B(n^{3-4p}) - \frac{n^{2p'-1}}{\sqrt{3-4p'}}B'(n^{3-4p'}) + o(\sqrt{n\log\log n}) \hspace{0.5em} \textit{a.s.}
\end{equation*}
Now we consider the Gaussian process 
\begin{equation*}
X(t) := \frac{t^{2p-1}}{\sqrt{3-4p}}B(t^{3-4p}) -
 \frac{t^{2p'-1}}{\sqrt{3-4p'}}B'(t^{3-4p'}). 
\end{equation*}
Since its covariance kernel is given by 
\[
 R(s,t) = \frac{(st)^{2p-1}}{3-4p} (s\wedge t)^{3-4p}
+ \frac{(st)^{2p'-1}}{3-4p'} (s\wedge t)^{3-4p'}, 
\]
we see that $R(s,t)$ satisfies (\ref{assv}) with $\rho=1/2$, and 
\[
 h(x)=\frac{x^{2p-\frac{3}{2}}}{3-4p}+\frac{x^{2p'-\frac{3}{2}}}{3-4p'}
 = O(x^{-(\frac{3}{2}-2 (p \vee p')}) \quad \text{as $x \to
 \infty$}, 
\]
which implies (\ref{h order of convergence}) for $0< p \vee p' <
 3/4$. Therefore, we obtain \eqref{eq:asymp_variance} from Theorem~\ref{cLIL main result}. 
\end{proof}

\section{Upper and lower bounds for the LIL}

In this section, let $\{X(s)\}_{s \ge 0}$ be a centered continuous Gaussian
 process satisfying (\ref{assv}). Then, we begin by noting that Theorem~\ref{cLIL main result} follows from Propositions~\ref{prop:upperLIL} and 
 \ref{prop:lowerLIL} provided that condition
 \eqref{eq:subseq} below holds. This condition will be verified in the next section. 

\subsection{Upper bound for the LIL}
First we show the upper bound in the LIL,  
which is the simpler part of the proof. 

\begin{prop}\label{prop:upperLIL} 
Let $\{X(s)\}_{s \ge 0}$ be a centered continuous Gaussian
 process satisfying (\ref{assv}). Then, we have  
\[
 \limsup_{s \to \infty}  \frac{X(s)}{\sqrt{2s^{2\rho}  \log \log s}} 
\le \sigma \quad\text{a.s.},
\] 
where $\sigma:=R(1,1)^{1/2}$.
\end{prop}

We recall below the well-known Borell–TIS inequality, which provides an estimate on the tail probability of the supremum of a Gaussian process
(cf. \cite[Theorem~4.2]{Nourdin2012}).

\begin{thm}\label{thm:tail}
Let $Y=\{Y(s)\}_{0\le s\le 1}$ be a centered, continuous Gaussian process. 
Set $v := \sup_{0\le s\le 1} \var(Y(s))$. 
Then, $m:= \mathbb{E}\left[\sup_{0\le s\le 1}
 Y(s)\right]$ is finite and we have, for all $x>m$, 
    \begin{equation*}
        \mathbb{P}\left(\sup_{0\le s\le 1} Y(s) \geq x\right)
	 \leq
	 \exp\left(-\frac{(x-m)^2}{2v}\right). 
    \end{equation*}
\end{thm}

\begin{proof}[Proof of Proposition~\ref{prop:upperLIL}] 

By (\ref{assv}) and $\rho$-self-similarity of $\{X(s)\}_{s \ge 0}$, 
it follows from Theorem~\ref{thm:tail} that 
for any $0<a<b$, 
\begin{align}
\P(\max_{a \le s \le b} X(s) > b^{\rho} M) 
&\le \P(\max_{0 \le s \le 1} b^{-\rho} X(bs) > M)
 \nonumber \\
&= \P(\max_{0 \le s \le 1} X(s) > M) \nonumber \\
&\le \exp\Big( -\frac{(M-m)^2}{2 v} \Big),  
\label{Nourdin}
\end{align}
where $m=\E[\sup_{0\le s\le 1} X(s)]$ and 
\begin{equation}
    v = \sup_{0\le t\le 1}
 \var(X(t)) = \sup_{0\le t\le 1} (t^{2\rho} R(1,1))  =\sigma^2. 
\label{eq:v=sigma2}
\end{equation}
Let $t_n=\alpha^n$ for fixed $\alpha>1$ and 
set 
\[
M_n := \sigma \sqrt{2\alpha \log \log t_n} 
= \sigma \sqrt{2\alpha (\log n + \log \log \alpha)}. 
\]
For any $\epsilon>0$, we have 
\begin{equation}
 \frac{(M_n - m)^2}{2\sigma^2} \ge 
 (1-\epsilon) \alpha \log n 
\label{eq:estimate1} 
\end{equation}
for any sufficiently large $n$. 
Setting $b=t_{n+1}$ and $M=M_n$ in \eqref{Nourdin}
 together with \eqref{eq:estimate1} and \eqref{eq:v=sigma2} yields 
\begin{align*}
\P\left( \max_{t_n \le s \le t_{n+1}} X(s) > t_{n+1}^{\rho}
 M_n \right) 
&\le \exp\Big( -\frac{(M_n-m)^2}{2
 \sigma^2} \Big)\\ 
&\le n^{-(1-\epsilon)\alpha}. 
\end{align*}
Hence, if we put $\alpha=(1-2\epsilon)^{-1}$, then 
\[
 \sum_{n=1}^{\infty}
\P\left(\max_{t_n \le s \le t_{n+1}} X(s) > t_{n+1}^{\rho} 
M_n \right) < \infty.  
\]
By the Borel-Cantelli lemma, we have, almost surely, 
\[
\max_{t_n \le s \le t_{n+1}} X(s) \le t_{n+1}^{\rho} M_n 
\quad \text{for all $n \ge n_0$}. 
\]
Therefore, for $t_n\le s\le t_{n+1}$, 
\[
 \frac{X(s)}{\sqrt{2s^{2\rho} \log \log s}}
\le \frac{t_{n+1}^{\rho}
 M_n}{\sqrt{2t_n^{2\rho} \log \log t_n}}
= \alpha^{\rho+\frac{1}{2}}\sigma \quad a.s. 
\]
Since $\alpha>1$ is arbitrary, it follows that 
\[
\limsup_{s \to \infty}  \frac{X(s)}{\sqrt{2s^{2\rho} \log \log s}} 
\le \sigma \quad a.s.
\]
\end{proof}

\subsection{Lower bound for the LIL}

We now proceed to establish the lower bound in the LIL.
\begin{prop}\label{prop:lowerLIL} 
Let $\{X(s)\}_{s \ge 0}$ be a centered continuous Gaussian
 process satisfying (\ref{assv}) and 
\begin{equation}
   \P\left(X(t_{n+1})-X(t_n) > \left(2\gamma_n\log\log
			t_{n+1}\right)^{1/2} \ i.o.\right)=1, 
\label{eq:subseq}
\end{equation}
where $t_n = \alpha^n$ for any sufficiently large $\alpha$ and $\gamma_n = \var(X(t_{n+1}) -
 X(t_n))$. 
Then, we have
\begin{equation}
 \limsup_{s \to \infty}  \frac{X(s)}{\sqrt{2s^{2\rho} \log \log s}} 
\geq \sigma\quad a.s.,  
\label{eq:lowerbound} 
\end{equation}
where $\sigma:=R(1,1)^{1/2}$.
\end{prop}
\begin{proof}
First we note that 
\begin{align}
\gamma_n &= R(t_{n+1},t_{n+1})-2R(t_{n+1},t_n)+R(t_n,t_n). 
\label{eq:gamma_n}
\end{align}
By the Cauchy–Schwarz inequality and \eqref{assv}, we see that 
\begin{align*}
|R(t_{n+1},t_n)| 
\leq R(t_{n+1},t_{n+1})^{1/2}
 R(t_n,t_n)^{1/2} = \sigma^2 (t_{n+1}t_n)^{\rho}, 
\end{align*}
which, together with
 $t_{n}^{\rho}=\alpha^{-\rho}t_{n+1}^{\rho}=o_{\alpha}(t_{n+1}^{\rho})$, 
implies 
\begin{equation}\label{the order of gamma_n}
    \gamma_n = \sigma^2 t_{n+1}^{2\rho}(1+o_{\alpha}(1)) \quad  \text{as $n \to \infty$}.
\end{equation} 
Dividing both sides of the inequality in \eqref{eq:subseq} by
 $\sqrt{2t_{n+1}^{2\rho}\log\log t_{n+1}}$ yields 
\begin{equation*}
    \frac{X(t_{n+1})}{\sqrt{2t_{n+1}^{2\rho}\log\log
     t_{n+1}}} >
     \left(\frac{\gamma_n}{t_{n+1}^{2\rho}}\right)^{1/2} +
     \frac{X(t_{n})}{\sqrt{2t_{n+1}^{2\rho}\log\log
     t_{n+1}}}. 
\end{equation*}
Since 
\begin{equation*}
   \limsup_{n \to \infty} \frac{X(t_n)}{\sqrt{2t_{n}^{2\rho}
    \log \log t_n}} \leq \sigma \quad a.s. 
\end{equation*}
follows from Proposition~\ref{prop:upperLIL}, by using \eqref{the order of gamma_n}, 
we obtain 
\begin{equation*}
    \limsup_{n \to
     \infty}\frac{X(t_{n+1})}{\sqrt{2t_{n+1}^{2\rho}\log\log
     t_{n+1}}} \geq 
\sigma 
- \frac{\sigma}{\alpha^{\rho}} \quad a.s.
\end{equation*}
Taking the limit $\alpha \to \infty$ yields \eqref{eq:lowerbound}.  
\end{proof}

\section{Proof of \eqref{eq:subseq} via the Erd\H{o}s-R\'{e}nyi's Borel-Cantelli lemma} 

To prove \eqref{eq:subseq}, we use the following version of
the Borel-Cantelli lemma, proved by Erd\H{o}s and R\'{e}nyi. 

\begin{thm}[Erd\H{o}s and R\'{e}nyi \cite{erdHos1959cantor}]\label{thm:weak indep}
Let $A_n$ be a sequence of events such that $\sum_{n=1}^\infty \mathbb{P}\left(A_n\right) = \infty$ and
\begin{equation}\label{cond:weak indep}
  \liminf_{n\to\infty}\frac{\sum_{k=1}^n 
\sum_{l=1}^n\left( \mathbb{P}\left(A_k\cap A_l\right) 
-\mathbb{P}\left(A_k\right)\mathbb{P}\left(A_l\right)\right)}{\left(\sum_{k=1}^n
  \mathbb{P}\left(A_k\right)\right)^2} = 0. 
\end{equation}
Then, $\mathbb{P}\left(A_n\ i.o.\right) = 1$.
\end{thm}

We denote $t_n = \alpha^n \ (\alpha >1)$ as in
Proposition~\ref{prop:lowerLIL}, 
\begin{equation}\label{a_n}
a_k:= (2 \log \log t_{k+1})^{1/2}= (2(\log (k+1)
    +\log\log \alpha))^{1/2} 
\end{equation}
and consider the event
\begin{equation}
A_k:= \{X(t_{k+1})-X(t_k) > \gamma^{1/2}_k a_k \}. 
\label{eq:eventAn}
\end{equation} 

First we note that since $\chi_k:=(X(t_{k+1})-X(t_k))/\gamma_k^{1/2}$ is standard
 normal, it follows from the bounds 
\begin{equation*}
        \frac{1}{2}x^{-1} e^{-\frac{x^2}{2}} \leq
	 \int_x^\infty e^{-\frac{y^2}{2}}dy \leq x^{-1}
	 e^{-\frac{x^2}{2}} \quad (x \ge 1)
\end{equation*}
that 
\begin{equation}
 \P(A_k) = \P(\chi_k > a_k) 
= \Theta_{\alpha}\Big( 
\frac{1}{k\sqrt{\log k}}
\Big).
\label{eq:bounds4Ak} 
\end{equation}
Here we write $c_k = \Theta_{\alpha}(d_k)$ to mean that the ratio $c_k
/ d_k$ is bounded above and below uniformly in $k=2,3,\dots$ by positive constants
depending only on $\alpha$. 

Now we estimate the numerator and denominator of
\eqref{cond:weak indep} separately. 
The following estimate for the denominator follows immediately from 
\eqref{eq:bounds4Ak}.  

\begin{lem}\label{lem:lowerLIL:the first half}
Under the settings as above, we have
\begin{equation}\label{2-bc:suminf}
    \sum_{k=1}^n \mathbb{P}\left(A_k\right) 
= \Theta_{\alpha}\big( \sqrt{\log n} \big). 
\end{equation}
\end{lem}
\begin{proof}
By summing both sides of \eqref{eq:bounds4Ak}, we obtain the
inequality \eqref{2-bc:suminf}. 
\end{proof}

Next we estimate the numerator. 
Let $R_{a,b} := [a,\infty) \times [b,\infty)$ and 
$Z_{\delta} \sim N(0,I_{\delta})$, where $I_{\delta}=\bigl(
\begin{smallmatrix}
   1 & \delta \\
   \delta & 1
\end{smallmatrix}
\bigl)$ is the covariance matrix of a bivariate normal distribution. We observe that 
\begin{align*}
&\P(A_k \cap A_l) - \P(A_k) \P(A_l)\\
&= \P(\{\chi_k > a_k\} \cap \{\chi_l > a_l\})
- \P(\chi_k >
 a_k)\P(\chi_l > a_l) \\
&= \P(Z_{\delta_{k,l}} \in R_{a_k,a_l}) - \P(Z_0 \in
 R_{a_k,a_l}), 
\end{align*} 
where $\delta_{k,l}$ denotes the covariance 
between $\chi_k$ and $\chi_l$, or equivalently, the
correlation coefficient between the increments
$X(t_{k+1})-X(t_k)$ and $X(t_{l+1})-X(t_l)$. 
To proceed, we require an estimate for the tail probability of
the bivariate normal vector $Z_{\delta}$ over the quadrant $R_{a,b}$. 
In this context, we have the following remarkable formula for the tail distribution
function. 
\begin{prop} 
Let us define 
\[
\psi(\delta,x,y) := \frac{1}{2\pi\sqrt{1-\delta^2}}  
\exp\left(-\frac{x^2 - 2 \delta x y +
 y^2}{2(1-\delta^2)}\right),  
\]
which is the probability density function of the bivariate normal
random variable $Z_{\delta} \sim N(0,I_{\delta})$. 
For $a, b \in \R$ and $|\delta|<1$, we have 
\begin{align*}
\phi(\delta,a,b) := \P(Z_{\delta} \in R_{a,b}) - \P(Z_0 \in R_{a,b})
= \int_0^{\delta} \psi(t,a,b) dt. 
\end{align*}
\end{prop}
\begin{proof} 
It is easy to verify that 
\begin{equation}
 \partial_{\delta} \psi(\delta,x,y) 
= \partial_x \partial_y \psi(\delta,x,y). 
\label{eq:diff_identity} 
\end{equation}
Differentiating $\phi(\delta,a,b)$ with respect to $\delta$ 
and using identity \eqref{eq:diff_identity}, we obtain 
\begin{align*}
 \partial_{\delta} \phi(\delta,a,b)
&= \int_a^{\infty} dx 
\int_b^{\infty} dy 
 \partial_{\delta} \psi(\delta,x,y)\\ 
&= \int_a^{\infty} dx 
\int_b^{\infty} dy 
\partial_x \partial_y \psi(\delta,x,y) \\
&= \psi(\delta,a,b).  
\end{align*}
Since $\phi(0,a,b)=0$, integrating both sides with respect
 to $\delta$ yields the desired formula.  
\end{proof}

From this proposition, we have the following estimate for
$\phi(\delta,a,b)$.

\begin{lem}\label{lem:upper bound of the difference}
For $a,b>0$ and $|\delta|<1$,
\begin{align*}
|\phi(\delta,a,b)|\le
\frac{1}{2\pi} \frac{1}{\sqrt{1-\delta^2}}\exp\left(-\frac{a^2 + b^2}{2}\right) 
  \exp\left(\frac{|\delta| a b}{1-\delta^2}\right)|\delta|.
\end{align*} 
\end{lem}
\begin{proof}
Now we estimate the following: 
\[
 \phi(\delta,a,b) 
= \int_0^{\delta} \frac{1}{2\pi \sqrt{1-t^2}}  
\exp\left(-\frac{a^2 - 2 t a b +
 b^2}{2(1-t^2)}\right) dt.  
\]
For $|\delta|< 1$, we obtain
\begin{align*}
|\phi(\delta,a,b)|
&\le 
\frac{1}{2\pi} \frac{1}{\sqrt{1-\delta^2}}  
\exp\left(-\frac{a^2 + b^2}{2}\right) 
\left| \int_0^{\delta} 
\exp\left(\frac{t a b}{1-t^2}\right) dt
\right|, 
\end{align*}
from which the desired inequality follows. 
\end{proof}

\begin{lem}
Let $\{X(s)\}_{s \ge 0}$ be a centered continuous Gaussian
 process satisfying \eqref{assv}, and set 
\begin{equation*}
\gamma_{k,l} 
= \cov(X(t_{k+1})-X(t_k), X(t_{l+1})-X(t_l)), 
\end{equation*}
and $\gamma_k=\gamma_{k,k}$, which is the same as in \eqref{eq:gamma_n}. 
Then, 
\begin{equation}\label{eq:delta=ratioofL}
    \delta_{k,l} :=
     \frac{\gamma_{k,l}}{\gamma_k^{\frac{1}{2}}
     \gamma_l^{\frac{1}{2}}}
= \frac{L_{|k-l|}(\alpha)}{L_0(\alpha)},  
\end{equation}
where 
$L_j(\alpha) 
:= h(\alpha^j) -\alpha^{-\rho} \{h(\alpha^{j+1})
+ h(\alpha^{|j-1|})\}
+ \alpha^{-2\rho} h(\alpha^j)$. 
\end{lem}

\begin{proof} 
Since $R(x,y) = (xy)^{\rho} h(xy^{-1})$ when $x \ge y$, 
we easily see from \eqref{assv} that 
\begin{align*}
 \gamma_{k,l} 
&= R(t_{k+1}, t_{l+1}) - R(t_{k+1}, t_{l}) -
 R(t_{k}, t_{l+1}) + R(t_{k}, t_{l}) \\
&= \alpha^{\rho(k+l+2)} L_{|k-l|}(\alpha), 
\end{align*}
from which \eqref{eq:delta=ratioofL} immediately follows. 
\end{proof}

In the following lemma, let $a_k$ be defined as in
\eqref{a_n}, and for simplicity, 
we write $\delta_j$ to denote $\delta_{k,l}$ when $|k-l|=j$. 

\begin{lem}\label{between |k-l| and C}
Suppose that \eqref{assv} and \eqref{h order of convergence}
 hold. 
Then, for $0<\epsilon<\frac{1}{2}(\eta \wedge
 1)$ and sufficiently large $\alpha$, there exist $C_{\eta, \alpha}>0$ and
 $M_{\eta,\epsilon,\alpha}>0$ such that 
\begin{equation*}
  |\phi(\delta_j, a_k,a_{k+j}) | \le C_{\eta, \alpha}\frac{1}{k^{1+\epsilon}}\frac{1}{j^{1+\eta-2\epsilon}}
\end{equation*}
for any $j\ge M_{\eta,\epsilon,\alpha}$. 
\end{lem}
\begin{proof}
From \eqref{h order of convergence},
 (\ref{eq:delta=ratioofL}), 
and the definition of $L_j(\alpha)$, 
it follows that $L_0(\alpha) \to h(1) > 0$ as $\alpha \to \infty$
 and for any sufficiently large $\alpha$, 
\begin{equation}\label{delta from ass2}
  |\delta_j|=\left|\frac{L_{j}(\alpha)}{L_{0}(\alpha)}\right|
\le c_{\eta} (j \log \alpha)^{-\eta} \quad (j \ge 2). 
\end{equation}
It follows from Lemma~\ref{lem:upper bound of the difference} that
\begin{align}\label{phi delta estimate}
 |\phi(\delta_j, a_k, a_{k+j}) |&\le C_{\eta, \alpha}\frac{1}{\sqrt{1-\delta_{j}^2}}\frac{1}{k+1}\frac{1}{k+j+1}  
\exp\left(\frac{|\delta_j| a_k
 a_{k+j}}{1-\delta_j^2}\right)j^{-\eta}.
\end{align}
For any fixed $\epsilon$ with 
$0<\epsilon<\frac{1}{2}(\eta \wedge 1)$, by (\ref{delta from ass2}), we
 have $\frac{1}{\sqrt{1-\delta_{j}^2}}\le\frac{1}{2}$ and $|\delta_j|\le \epsilon$
for sufficiently large $j$. Thus, we see that 
\begin{align*}
    \frac{|\delta_j|}{1-\delta_j^2}a^2_{k+j}&=
 \frac{|\delta_j|}{1-\delta_j^2}\times 2(\log (k+j+1)+\log
 \log \alpha) \notag \\ 
    &\le \epsilon\log (k+j+1) 
\end{align*}
holds 
for any $j\ge M_{\eta, \epsilon,\alpha}$, 
and hence, 
\begin{align*}
      \exp\left(\frac{|\delta_j| a_k
 a_{k+j}}{1-\delta_j^2}\right)&\le
 \exp\left(\frac{|\delta_j|
 a^2_{k+j}}{1-\delta_j^2}\right) 
   \le (k+j+1)^{\epsilon}. 
\end{align*}
It follows from (\ref{phi delta estimate}) that 
\begin{align*}
|\phi(\delta_j, a_k,a_{k+j}) | 
&\le C_{\eta, \alpha}\frac{1}{k+1}\frac{1}{k+j+1}  
(k+j+1)^{\epsilon}j^{-\eta}\\ 
&\le C_{\eta, \alpha}\frac{1}{k^{1+\epsilon}}\frac{1}{j^{1+\eta-2\epsilon}}.
\end{align*} 
This completes the proof. 
\end{proof}

\begin{lem}\label{lem:lowerLIL:the latter part}
Suppose that \eqref{assv} and \eqref{h order of convergence} hold. Then, 
\begin{equation}\label{B-numerator}
  \sum_{k=1}^n \sum_{l=1}^n\left( \mathbb{P}\left(A_k\cap A_l\right)-\mathbb{P}\left(A_k\right)\mathbb{P}\left(A_l\right)\right) = O(\sqrt{\log n})\text{\quad as $n\to \infty$}.
   \end{equation}
\end{lem}
\begin{proof}
We divide the summation into 
\begin{align*}
\sum_{k,l= 1}^n 
&= \sum_{k,l= 1 \atop{|k-l| < M_{\eta, \epsilon,\alpha}}}^{n} 
+ \sum_{k,l= 1 \atop{|k-l| \ge M_{\eta, \epsilon,\alpha}}}^{n} 
=: I_1(n) + I_2(n).
\end{align*}
As for $I_1(n)$, from \eqref{eq:bounds4Ak}, we see that 
\begin{align*}
 |I_1(n)| 
&\le \sum_{k,l= 1 \atop{|k-l| < M_{\eta, \epsilon,\alpha}}}^n \P(A_k)\\
&\le 2 M_{\eta, \epsilon,\alpha}\left(1+c_{\alpha}\sum_{k=
 2}^{n}\frac{1}{k}\frac{1}{\sqrt{\log k}}\right)\\ 
&= O(\sqrt{\log n}) \quad \text{as $n\to \infty$}.
\end{align*}
On the other hand, it follows from Lemma~\ref{between |k-l| and C} that
\begin{align*}
   |I_2(n)| 
&\le 2 \sum_{k= 1}^{n}\sum_{{j \ge M_{\eta,
 \epsilon,\alpha}}} 
C_{\eta, \alpha}\frac{1}{k^{1+\epsilon}}\frac{1}{j^{1+\eta-2\epsilon}}=O(1)\text{\quad
 as $n\to \infty$}. 
   \end{align*}
Therefore, we obtain \eqref{B-numerator}. 
\end{proof}

Finally, we turn to the proof of \eqref{eq:subseq}. 
\begin{proof}[Proof of \eqref{eq:subseq}] 
Erd\H{o}s-R\'{e}nyi's Borel-Cantelli lemma (Theorem~\ref{thm:weak
 indep}), together with Lemmas~\ref{lem:lowerLIL:the
 first half} and \ref{lem:lowerLIL:the latter part}, yields \eqref{eq:subseq}. 
\end{proof}

As a final remark, we provide the proof of  Corollary~\ref{cor}.
\begin{proof}[Proof of Corollary~\ref{cor}]
 In the cases $p=p'$ and $0<p\vee p'<3/4$, the result
 follows from Theorem~\ref{asymptotic behavior the distance
 with same parameters} and Theorem~\ref{asymptotic behavior
 the distance with different parameters},
 respectively. Therefore, it remains to verify the two
 cases: (i) $p'<p$ and $p>3/4$, and (ii) $p'<p=3/4$.  
\begin{itemize}
 \item [(i)] If $p'<p$ and $p>3/4$, then it follows from (\ref{superdiffusive L}) that
\[
        \lim_{n\to \infty}\frac{|S_n-S_n'|}{n^{2p-1}} 
        =|L_{p,q}|>0 \quad \text{a.s.} 
\]
This means that the distance between two ERWs diverges to
       infinity, i.e., they meet each other only finitely often almost surely. 
 \item [(ii)] Similarly, if $p'<p=3/4$, we obtain the
       following result from the
       LIL for ERWs (cf. \cite{bercu2017martingale})
\[
\limsup_{n\to \infty}\pm\frac{S_n-S_n'}{\sqrt{2n \log
       n \log \log \log n}} = 1 \quad \text{a.s.} 
\]
This means that they meet each other infinitely often almost surely. 
\end{itemize}
\end{proof}

\textbf{Acknowledgment.} 
The authors would like to thank the anonymous referee for their careful reading and valuable comments.
TS was supported by JSPS KAKENHI Grant
Numbers JP22H05105, JP23H01077 and JP23K25774, and also supported in part
JP21H04432 and JP24KK0060. SS was supported by WISE program (MEXT) at Kyushu
University, and also supported in part JP23K25774 and the Kyushu University Fund ``Human Resource Development Initiative in Mathematics for Industry''. 

\bibliographystyle{plain}
\bibliography{reference}

\end{document}